\documentclass[12pt]{article}
\usepackage{graphicx}
\usepackage{epstopdf}
\usepackage{a4wide,color,hyperref} 
\usepackage{amsmath,amssymb,amsthm}

\newtheorem{theorem}{Theorem}
\newtheorem{lemma}[theorem]{Lemma}
\newtheorem{corollary}[theorem]{Corollary}

\title{Hamiltonian properties in generalized lexicographic products}
\author{Jan Ekstein\thanks{Department of Mathematics and European Centre of Excellence NTIS - New Technologies for the Information Society, Faculty of Applied Sciences, University of West Bohemia, Pilsen, Technick\'a 8, 306 14 Plze\v n, Czech Republic
\newline e-mail: \texttt{ekstein@kma.zcu.cz, teska@kma.zcu.cz}.}\and
Jakub Teska \footnotemark[1]}
\date{\today}

\begin{document}
\maketitle

\begin{abstract}
The lexicographic product $G[H]$ of two graphs $G$ and $H$ is obtained from $G$ by replacing each vertex with a copy of $H$ and adding all edges between any pair of copies corresponding to adjacent vertices of $G$. We consider also the generalized lexicographic product such that we replace each vertex of $G$ with arbitrary graph on the same number of vertices. We present sufficient and necessary conditions for traceability, hamiltonicity and hamiltonian connectivity of $G[H]$ if $G$ is a path and hence we improved and extended results in M. Kriesell, A Note on Hamiltonian Cycles in Lexicographical Products.
\end{abstract}

\section{Introduction}
A product of graphs is well known graph operation (e.g. Cartesian, direct, lexicographic) and study hamiltonian properties in some product of graphs is standard problem in graph theory. In this paper we denote to a lexicographic product of graphs. 

The lexicographic product $G[H]$ of two graphs $G$ and $H$ is defined by a vertex set $V(G[H])=V(G)\times V(H)$ and an edge set $E(G[H])=\{(g,h)(g',h'):gg'\in E(G)\mbox{ or } g=g'\wedge hh'\in E(H)\}$. In other words the lexicographic product $G[H]$ of two graphs $G$ and $H$ is obtained from $G$ by replacing each vertex with a copy of $H$ and adding all edges between any pair of copies corresponding to adjacent vertices of $G$. A typical sufficient condition for the existence of a hamiltonian cycle or a hamiltonian path in a lexicographic product $G[H]$ forces $G$ to contain a hamiltonian cycle or a hamiltonian path and $H$ to have some additional properties. Hamiltonian cycles and paths in lexicographic products have been studied in \cite{BarSz}, \cite{KaiKri}, \cite{Kri}, \cite{Ng}, and \cite{Tei}.

Clearly, $G[H]$ contains a cycle of length 3 if both of $G$ and $H$ contain at least one edge. Kaiser and Kriesell proved in \cite{KaiKri} that concepts of pancyclicity and hamiltonicity coincide in the case of lexicographic products of graphs with at least one edge. Recall that the graph $G$ is weakly pancyclic or pancyclic, if it contains cycles of every length between the length of a shortest cycle and that of a longest one, hamiltonian, respectively.

\begin{theorem}\emph{\textbf{\cite{KaiKri}}}
 \label{Pancyclicity}
 If $G$, $H$ are graphs with at least one edge each, then $G[H]$ is weakly pancyclic.
\end{theorem}  

In this paper we consider also the concept of the generalized lexicographic product mentioned in \cite{GuHou}, \cite{ChoKar} (defined as an expansion), and \cite{Sam}. Basically, the generalized lexicographic product is the graph $G[H_1, H_2,...,H_m]$, which will be a graph like a lexicographic product with the difference that every vertex of~$G$ can be replaced by a different graph $H_i$. Precisely, let $G$ be a graph with $V(G)=\{u_1,u_2,...,u_{m}\}$ and $H_i$ be an arbitrary graph $i=1,2,...m$. Then generalized lexicographic product $G[H_1, H_2,...,H_m]$ of a graph $G$ and $H_1, H_2,...,H_m$ is obtained from $G$ by replacing each vertex $u_i$ with the graph $H_i$ and adding all edges between graphs $H_i$ and $H_j$ if the corresponding vertices $u_i$, $u_j$ are adjacent in $G$.  We say that $G[H_1, H_2,...,H_m]$ is lex-regular if the number of vertices of $H_i$ is the same for $i=1,2,...,m$. 

For a given graph, $G$ we define $\pi(G)$ to be the maximum number of edges of a spanning linear forest of $G$ (a forest is linear if  its components are paths). 

Main results of this paper are the following theorems which generalize and improve some results from \cite{Kri}.

\begin{theorem}
 \label{main-odd}
Let $P_{2k+1}$ be a path with odd number of vertices, $k\geq 1$. Let $H_1,H_2,...,H_{2k+1}$ be graphs with $n$ vertices. The graph $P_{2k+1}[H_1, H_2,...,H_{2k+1}]$ is

\begin{itemize}

\item[(i)] hamiltonian if and only if $\pi(H_1)\geq 1$, $\pi(H_{2k+1})\geq 1$, and \newline $\sum^k_{i=0}\pi(H_{2i+1})\geq n$.

\item[(ii)] traceable if and only if $\sum^k_{i=0}\pi(H_{2i+1})\geq n-1$.

\item[(iii)] hamiltonian connected if and only if $\pi(H_1)\geq 2$, $\pi(H_{2k+1})\geq 2$, and \newline $\sum^k_{i=0}\pi(H_{2i+1})\geq n+1$.  

\end{itemize}
\end{theorem}

\begin{theorem}
 \label{main-even}
Let $P_{2k}$ be a path with even number of vertices, $k\geq 1$. Let $H_1,H_2,...,H_{2k}$ be graphs with $n$ vertices. Then the graph $P_{2k}[H_1, H_2,...,H_{2k}]$ is 

\begin{itemize}

\item[(i)] hamiltonian if and only if $k=1$ or $\pi(H_1)\geq 1$ and $\pi(H_{2k})\geq 1$.

\item[(ii)] traceable.

\item[(iii)] hamiltonian connected if and only if
      \begin{itemize}
         \item[$\bullet$] $\pi(H_1)\geq 1$ and $\pi(H_{2k})\geq 1$ for $k=1$.
         \item[$\bullet$] $\pi(H_1)\geq 2$ and $\pi(H_{2})\geq 2$ for $k>1$.
      \end{itemize}       

\end{itemize}
\end{theorem}

\bigskip

Observe that Theorem \ref{main-odd} and Theorem \ref{main-even} give a complete characterization of hamiltoni\-city of $P_n[H]$, traceability of $P_n[H]$, and hamiltonian connectivity of $P_n[H]$.

\section{Preliminaries}
As for standard terminology, we refer to the book by Bondy and Murty \cite{Bon}. However, before proving Theorem \ref{main-odd} and Theorem \ref{main-even} we mention several concepts and results which we need to make use of. 

For a multigraph $G$ and $x,y \in V(G)$ let $[x,y]_G$ be the set of edges between $x$ and $y$ and let $m_G(xy)=|[x,y]_G|$ be the multiplicity of the edge $xy$ in $G$. In particular, $\ell_G(x)=|[x,x]_G|$ denotes the number of loops at $x$ and $\ell(G)=\max\{\ell_G(x): x\in V(G)\}$. Note that the degree of a vertex $x$ denoted by $d_G(x)=\sum_{y\in V(G)} m_G(xy) + 2\ell_G(x)$. A multigraph $G$ is $k$ regular if $d_G(x)=k$ for every vertex $x$ in $V(G)$. Moreover, let $|G|=|V(G)|$ and $||G||=|E(G)|$. For any $X\subseteq V(G)$, let $G(X)$ be the submultigraph induced by $X$.

A multigraph $G'$ is said to be a multiple of a graph $G$ if $V(G')=V(G)$ and for all $x\neq y \in V(G)$, $m_{G'}(xy)>0$ holds only if $xy\in E(G)$. This means that from a given graph $G$, we can obtain a multiple $G'$ by adding loops or by replacing a single edge in $G$ by an arbitrary number of edges. 

In \cite{Kri}, Kriesell proved that $G[H]$ is hamiltonian if $G$ has a connected, $k$-regular multiple with additional properties.

\begin{theorem}\emph{\textbf{\cite{Kri}}}
 \label{multiple}
 Let $G$ and $H$ be graphs. If $G$ has a connected, $2|H|$-regular multiple $G'$ satisfying $\pi(H)\geq \ell_{G'}(x)$ for all $x\in V(G')$, then $G[H]$ contains a hamiltonian cycle that contains exactly $m_{G'}(xy)$ edges between $V(G(\{x\})[H])$ and $V(G(\{y\})[H])$ for all $x\neq y$ in $V(G)$.
\end{theorem}

By a \emph{$uv$-path} we mean a path from $u$ to $v$ in $G$. If a $uv$-path is hamiltonian, we call it a \emph{$uv$-hamiltonian path}. The graph $G$ is traceable, if $G$ contains a hamiltonian path. The graph $G$ is hamiltonian connected, if every two vertices of $G$ are connected by a hamiltonian path.   

Teichert in \cite{Tei} and also Kriesell as corollary of Theorem \ref{multiple} in \cite{Kri} proved the following.

\begin{theorem}\emph{\textbf{(\cite{Kri} and \cite{Tei}})}
 \label{KriTei}
Let $G$ and $H$ be graphs, $|G|\geq 2$. Suppose that $G$ contains a hamiltonian path.
\begin{itemize}
\item[(i)] If $|G|$ is even, then $G[H]$ is traceable.

\item[(ii)] If $|G|$ is even and $||H||\geq 1$, then $G[H]$ is hamiltonian.

\item[(iii)] If $|G|$ is odd and $\frac{|H|-1}{2} \leq \pi(H)$, then $G[H]$ is traceable.

\item[(iv)] If $|G|$ is odd and $\frac{|H|}{2}\leq \pi(H)$, then $G[H]$ is hamiltonian.      
\end{itemize}
\end{theorem}

Note that the first two statements are in some sense necessary. Let $G$ be only a path with $n$ vertices, i.e. $G=P_n$. Clearly, $P_2[H]$ is hamiltonian. If $n\geq 4$ is even, then $P_n[H]$ is hamiltonian if and only if $||H||\geq 1$ because if $H$ has no edge, then $P_n[H]$ cannot have a hamiltonian cycle.

\begin{figure}[ht]
\begin{center}
\includegraphics[width=6cm]{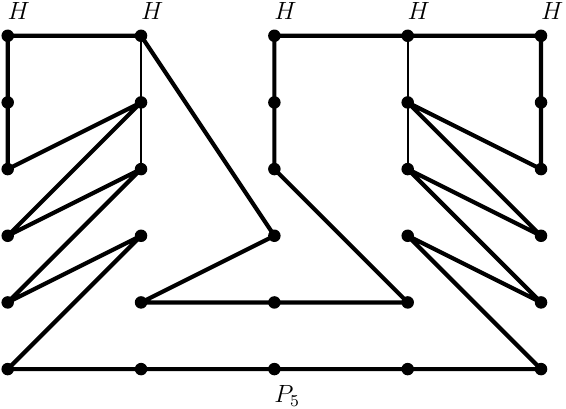}
\end{center}\caption{Hamiltonian cycle in $P_5[P_3+3K_1]$ (bold edges)}\label{P5}
\end{figure}  

\bigskip

But the last two statements are not necessary. For example, take the graph~$H$ on 6 vertices with $\pi(H)=2$ (e.g. the graph $H=P_3+3K_1$). Thus this graph does not satisfy the conditions $(iii)$ and $(iv)$ in the previous theorem. If $G=P_3$, then $G[H]$ is neither hamiltonian nor traceable (see the proof of Theorem \ref{main-odd}). But if we instead of $P_3$ take $P_5$ as a graph $G$, then $G[H]$ is hamiltonian (see Figure~\ref{P5}, edges of $P_5[P_3+3K_1]$ between consecutive copies of $H$ are missing for the clarity). For longer odd paths $G=P_{2k+1}$, $k\geq 3$, the lexicographic product $G[H]$ is also hamiltonian.  

\section{Proofs}
Let $P_{2k+1}$ be a path with odd number of vertices consecutively denoted by $u_1, u_2,...,u_{2k+1}$, $k\geq 1$, and edges denoted by $e_m$ where $e_m=u_mu_{m+1}$, $m=1,2,...,2k$.

\bigskip

\noindent
\textbf{\emph{Proof of Theorem \ref{main-odd}}}

\noindent
\textbf{\emph{(i)}}
~First suppose that $\pi(H_1)\geq 1$, $\pi(H_{2k+1})\geq 1$ and $\sum^k_{i=0}\pi(H_{2i+1})\geq n$. 
Clearly, we have $n\geq 2$ because of $\pi(H_1)\geq 1$. Now we find a connected $2n$-regular multiple of $P_{2k+1}$. Then we prove the hamiltonicity of $P_{2k+1}[H_1, H_2,...,H_{2k+1}]$ similarly as in \cite{Kri}.

We define the number of loops at each vertex $u_i$ of multiple $G'$ of $P_{2k+1}$. For even vertices $u_2,u_4,...,u_{2k}$ we define $\ell_{G'}(u_{2i})=0$, $i=1,2,...,k$, and for odd vertices $\ell_{G'}(u_{2i+1})=\pi(H_{2i+1})$, $i=0,1,...,k$. If $\sum^k_{i=0} \pi(H_{2i+1}) > n$ (the multiple $G'$ has more than $n$ loops), then we remove arbitrary loops from $G'$ in such a way that $\ell_{G'}(u_1)\geq 1$, $\ell_{G'}(u_{2k+1})\geq 1$ and $\sum^k_{i=0} \ell_{G'}(u_{2i+1}) = n$. Note that $\ell_{G'}(u_j) \leq \pi(H_j)$ for $j=1,2,...,2k+1$.

Now we define the multiplicity of every edge $e_m$ of $P_{2k+1}$, $m=1,2,...,2k$, 

$$m_{G'}(e_1)=2n-2\ell_{G'}(u_1)$$ 
$$m_{G'}(e_2)=2\ell_{G'}(u_1)$$
$$m_{G'}(e_3)=2n-2\ell_{G'}(u_1)-2\ell_{G'}(u_3)$$ 
$$m_{G'}(e_4)=2\ell_{G'}(u_1) + 2\ell_{G'}(u_3)$$
$$m_{G'}(e_5)=2n-2\ell_{G'}(u_1)-2\ell_{G'}(u_3)-2\ell_{G'}(u_5)$$ 
$$m_{G'}(e_6)=2\ell_{G'}(u_1) + 2\ell_{G'}(u_3)+2\ell_{G'}(u_5)$$

\noindent
in general we have, $i=1,2,...,k$:
$$m_{G'}(e_{2i})=\sum^i_{j=1} 2\ell_{G'}(u_{2j-1}) \mbox{~~ and ~~}
m_{G'}(e_{2i-1})=2n-\sum^i_{j=1} 2\ell_{G'}(u_{2j-1}).$$

Clearly, by the construction the multiplicity of every edge is at least 2, and the degree of every vertex $u_1, u_2,...,u_{2k}$ is exactly $2n$ and for the last vertex of our path we have: 

$$d_{G'}(u_{2k+1})= 2\ell_{G'}(u_{2k+1})+m_{G'}(e_{2k})= 2\ell_{G'}(u_{2k+1})+\sum^k_{j=1} 2\ell_{G'}(u_{2j-1})=$$$$2\sum^k_{i=0} \ell_{G'}(u_{2i+1}) = 2n.$$

Now we prove that the graph $P_{2k+1}[H_1,H_2,...,H_{2k+1}]$ contains a hamiltonian cycle using exactly $\ell_{G'}(u_i)$ edges of $H_i$ and exactly $m_{G'}(e_j)$ edges between \newline $V(P_{2k+1}(\{u_j\})[H_j])$ and $V(P_{2k+1}(\{u_{j+1}\})[H_{j+1}])$ for $i=1,2,...,2k+1$ and $j=1,2,...,2k$. 

For every vertex $u_i$ and every graph $H_i$, there exists a spanning linear subforest of $P_{2k+1}(\{u_i\})[H_i]\cong H_i$ with components $P_1(u_i),P_2(u_i),...,P_{j_i}(u_i)$ satisfying 

$$j_i=|H_i|-\ell_{G'}(u_i) \hskip 5mm \mbox{ and } \hskip 5mm \sum^{j_i}_{t=1}||P_t(u_i)||=\ell_{G'}(u_i),$$ because $\ell_{G'}(u_i) \leq \pi(H_i)$.

\begin{figure}[ht]
\begin{center}
\includegraphics[width=10cm]{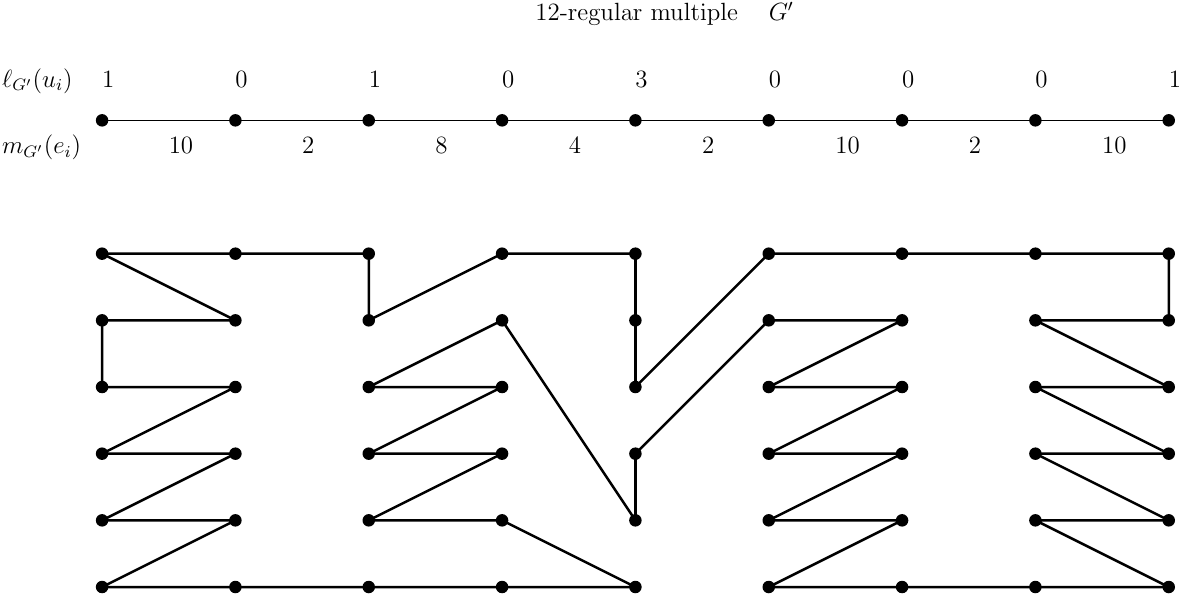}
\end{center}\caption{Hamiltonian cycle in $P_9[H_1,H_2,...,H_9]$ from a multiple $G'$}
\label{P9}
\end{figure} 

Futhermore, after removing all the loops from multiple $G'$, there exists a closed eulerian trail $C$ in the graph $G'$. We obtain the hamiltonian cycle of $P_{2k+1}[H_1, H_2,...,H_{2k+1}]$ as required by replacing simultaneously the vertices $u_i$ at their $t$-th occurence in $C$ by the component $P_t(u_i)$ for $t=1,2,...,j_i$ and $i=1,2,...,2k+1$ (for illustration see Figure \ref{P9}).   

\bigskip

Now suppose that $P_{2k+1}[H_1, H_2,...,H_{2k+1}]$ is hamiltonian. If $||H_1||=0$ or $||H_{2k+1}||=0$, then $P_{2k+1}[H_1, H_2,...,H_{2k+1}]$ cannot contain a hamiltonian cycle. Hence assume that $\pi(H_1) \geq 1$, $\pi(H_{2k+1}) \geq 1$ and $P_{2k+1}[H_1, H_2,...,H_{2k+1}]$ has a hamiltonian cycle $C$. Note that if $H_i$ has at most $\pi(H_i)$ edges in linear forest, then the number of components of a linear forest of $H_i$ is at least $n-\pi(H_i)$. Now we count the number of edges between graphs $H_1, H_2,...,H_{2k+1}$ in $C$. 

The graph $H_1$ has at least $n-\pi(H_1)$ components. Therefore there are at least $2(n-\pi(H_1))$ edges between $H_1$ and $H_2$ in $C$. 

Since $H_2$ has only $n$ vertices, there are at most $2n$ edges in $C$ from $H_2$. Thus, there are at most $2n-2(n-\pi(H_1))=2\pi(H_1)$ edges between $H_2$ and $H_3$ in $C$.

Again, $H_3$ has at least $n-\pi(H_3)$ components. Therefore there are at least $2(n-\pi(H_3))-2\pi(H_1)=2n-2\pi(H_3)-2\pi(H_1)$ edges between $H_3$ and $H_4$ in $C$. 

Since $H_4$ has only $n$ vertices, there are at most $2n$ edges in $C$ from $H_4$. Thus, there are at most $2n-(2n-2\pi(H_3)- 2\pi(H_1))=2\pi(H_3) + 2\pi(H_1)$ edges between $H_4$ and $H_5$ in $C$.

If we continue step by step, we get that between $H_{2k}$ and $H_{2k+1}$ there are at most $2\pi(H_{2k-1})+2\pi(H_{2k-3})+\cdots+2\pi(H_3)+2\pi(H_1)$ edges in $C$ and from the other side $H_{2k+1}$ has at least $n-\pi(H_{2k+1})$ components. Therefore there should be at least $2(n-\pi(H_{2k+1}))$ edges between $H_{2k+1}$ and $H_{2k}$ in $C$. Now we get that  

$$2(n-\pi(H_{2k+1}))\leq 2\pi(H_{2k-1})+2\pi(H_{2k-3})+\cdots+2\pi(H_3)+2\pi(H_1)$$ 

$$n\leq\pi(H_{2k+1})+\pi(H_{2k-1})+\pi(H_{2k-3})+\cdots+\pi(H_3)+\pi(H_1)$$

$$n\leq\sum^k_{i=0}\pi(H_{2i+1}).$$
Thus we finish the proof of Theorem \ref{main-odd} (i). $~~~~~~~~~~~~~~~~~~~~~~~~~~~~~~~~~~~~~~~~~~~~~~~~~~~~~~~~~~~~~~~~~~~ \square$

\vskip 2cm

Before the proofs of Theorem \ref{main-odd} statements (ii) and (iii), we define functions $A(t)$, $B(t)$ and state the following lemmas. Let $a,b,t\in\{1,2,...,2k+1\}$. 

 $$A(t)=0 \mbox{ for } t<a; ~~~ B(t)=0 \mbox{ for } t<b;$$
 $$A(t)=1 \mbox{ for } t\geq a; ~~~ B(t)=1 \mbox{ for } t\geq b.$$ 

\bigskip
 
\begin{lemma}
 \label{lemmapath}
Let $P_{2k+1}$ be a path with odd number of vertices, $k\geq 1$, $a,b\in\{1,2,...,2k+1\}$. Let $H_1,H_2,...,H_{2k+1}$ be graphs with $n$ vertices such that one of the following conditions holds
\begin{itemize}
 \item[(I)] $a,b$ are even and $\pi(H_1)\geq 2$, $\pi(H_{2k+1})\geq 2$ and $\sum^k_{i=0}\pi(H_{2i+1})=n+1$;
 \item[(II)] $a$ is odd, $b$ is even and $\pi(H_1)\geq 1$, $\pi(H_{2k+1})\geq 1$ and $\sum^k_{i=0}\pi(H_{2i+1})=n$;
 \item[(III)] $a,b$ are odd and $\pi(H_1)\geq 1$, $\pi(H_{2k+1})\geq 1$ and $\sum^k_{i=0}\pi(H_{2i+1})=n-1$; 
 
 moreover for $a=1$, $b=2k+1$ we have only $\sum^k_{i=0}\pi(H_{2i+1})=n-1$.
\end{itemize} 
 
Then there exists a connected multiple $G'$ of $P_{2k+1}$ such that $d_{G'}(u_l)=2n$ for $l\in\{1,2,...,2k+1\}\setminus \{a,b\}$ and either $d_{G'}(u_a)=d_{G'}(u_b)=2n-1$ if $a\neq b$ or $d_{G'}(u_a)=2n-2$ if $a=b$.   
\end{lemma}

\begin{proof} 
(I) $a,b$ are even.  We have $n\geq 3$ because of $\pi(H_1)\geq 2$.

Similarly, as in the previous proof we define the number of loops at each vertex $u_i$ of multiple $G'$ of $P_{2k+1}$. For even vertices $u_2,u_4,...,u_{2k}$ we define $\ell_{G'}(u_{2i})=0$, $i=1,2,...,k$, and for odd vertices $\ell_{G'}(u_{2i+1})=\pi(H_{2i+1})$, $i=0,1,...,k$. Note that we have $\ell_{G'}(u_1)\geq 2$, $\ell_{G'}(u_{2k+1})\geq 2$ and $\sum^k_{i=0} \ell_{G'}(u_{2i+1})=n+1$.

\bigskip

Now we define the multiplicity of every edge $e_m$ of $P_{2k+1}$, $m=1,2,...,2k$,
$$m_{G'}(e_1)=2n-2\ell_{G'}(u_1)+A(1)+B(1)$$ 
$$m_{G'}(e_2)=2\ell_{G'}(u_1)-A(2)-B(2)$$
$$m_{G'}(e_3)=2n-2\ell_{G'}(u_1)-2\ell_{G'}(u_3)+A(3)+B(3)$$
$$m_{G'}(e_4)=2\ell_{G'}(u_1)+2\ell_{G'}(u_3)-A(4)-B(4)$$

\noindent
in general we have:
$$m_{G'}(e_{2i})=\sum^i_{j=1} 2\ell_{G'}(u_{2j-1})-A(2i)-B(2i) \mbox{~~~~~ and }$$  
$$m_{G'}(e_{2i-1})=2n-\sum^i_{j=1} 2\ell_{G'}(u_{2j-1})+A(2i-1)+B(2i-1), \mbox{ for } i=1,2,...,k.$$

Clearly, by the construction, the multiplicity of every edge is at least 2 and the degree of every vertex of $G'$ is the following:
$$d_{G'}(u_{1})=2\ell_{G'}(u_1)+m_{G'}(e_1)=2\ell_{G'}(u_1)+2n-2\ell_{G'}(u_1)+A(1)+B(1)=2n.$$
Clearly, $A(1)=B(1)=0$.

\bigskip

$$d_{G'}(u_{2i})=m_{G'}(e_{2i-1})+m_{G'}(e_{2i})=$$
$$2n-\sum^i_{j=1} 2\ell_{G'}(u_{2j-1})+A(2i-1)+B(2i-1)+\sum^i_{j=1} 2\ell_{G'}(u_{2j-1})-A(2i)-B(2i)=$$
$$2n+A(2i-1)+B(2i-1)-A(2i)-B(2i),\mbox{ for } i=1,2,...,k.$$

\noindent
Then $d_{G'}(u_{2i})=2n$ for $2i\notin\{a,b\}$, $d_{G'}(u_{2i})=2n-1$ for $2i\in\{a,b\}$, $a\neq b$, and $d_{G'}(u_{2i})=2n-2$ for $2i=a=b$.
$$d_{G'}(u_{2i+1})=m_{G'}(e_{2i})+m_{G'}(e_{2i+1})+2\ell_{G'}(u_{2i+1})=
\sum^i_{j=1} 2\ell_{G'}(u_{2j-1})-$$$$A(2i)-B(2i)+2n - \sum^{i+1}_{j=1} 2\ell_{G'}(u_{2j-1})+A(2i+1)+B(2i+1)+2\ell_{G'}(u_{2i+1})=$$
$$2n-A(2i)-B(2i)+A(2i+1)+B(2i+1)=2n,\mbox{ for } i=1,2,...,k-1.$$

\noindent
Note that $A(2i)=A(2i+1)$, $B(2i)=B(2i+1)$.

$$d_{G'}(u_{2k+1})=m_{G'}(e_{2k})+2\ell_{G'}(u_{2k+1})=\sum^{k+1}_{j=1} 2\ell_{G'}(u_{2j-1})-A(2k)-B(2k)=$$
$$2(n+1)-A(2k)-B(2k)=2n.$$
Clearly, $A(2k)=B(2k)=1$.

Resulting $G'$ is the connected multiple of $P_{2k+1}$ as required.

\bigskip

\noindent
(II) $a$ is odd and $b$ is even. We have $n\geq 2$ because of $\pi(H_1)\geq 1$.

We define the number of loops at each vertex $u_i$ of multiple $G'$ of $P_{2k+1}$ as in (I) such that we have $\ell_{G'}(u_1)\geq 1$, $\ell_{G'}(u_{2k+1})\geq 1$ and $\sum^k_{i=0} \ell_{G'}(u_{2i+1})=n$.

\bigskip

Now we define the multiplicity of every edge $e_m$ of $P_{2k+1}$, $m=1,2,...,2k$,

$$m_{G'}(e_{2i})=\sum^i_{j=1} 2\ell_{G'}(u_{2j-1})+A(2i)-B(2i) \mbox{~~~~ and }$$  
$$m_{G'}(e_{2i-1})=2n-\sum^i_{j=1} 2\ell_{G'}(u_{2j-1})-A(2i-1)+B(2i-1), \mbox{ for } i=1,2,...,k.$$

Clearly, by the construction the multiplicity of every edge is at least 1 and the degree of every vertex of $G'$ is the following:
$$d_{G'}(u_{1})=2\ell_{G'}(u_1)+m_{G'}(e_1)=2\ell_{G'}(u_1)+2n-2\ell_{G'}(u_1)-A(1)+B(1)=2n-A(1).$$

\noindent
Clearly, $B(1)=0$, $d_{G'}(u_1)=2n$ for $a\neq 1$ and $d_{G'}(u_1)=2n-1$ for $a=1$.

\bigskip

$$d_{G'}(u_{2i})=m_{G'}(e_{2i-1})+m_{G'}(e_{2i})=$$
$$2n-\sum^i_{j=1}2\ell_{G'}(u_{2j-1})-A(2i-1)+B(2i-1)+\sum^i_{j=1}2\ell_{G'}(u_{2j-1})+A(2i)-B(2i)=$$
$$2n-A(2i-1)+B(2i-1)+A(2i)-B(2i)=$$$$2n+B(2i-1)-B(2i), \mbox{ for } i=1,2,...,k.$$

\noindent
Then $A(2i-1)=A(2i)$, $d_{G'}(u_{2i})=2n$ for $2i\neq b$ and $d_{G'}(u_{2i})=2n-1$ for $2i=b$.

\bigskip

$$d_{G'}(u_{2i+1})=m_{G'}(e_{2i})+m_{G'}(e_{2i+1})+2\ell_{G'}(u_{2i+1})=$$
$$\sum^i_{j=1} 2\ell_{G'}(u_{2j-1})+A(2i)-B(2i)+$$$$2n-\sum^{i+1}_{j=1} 2\ell_{G'}(u_{2j-1})-A(2i+1)+B(2i+1)+2\ell_{G'}(u_{2i+1})=$$ 
$$2n+A(2i)-B(2i)-A(2i+1)+B(2i+1)=$$$$2n+A(2i)-A(2i+1), \mbox{ for } i=1,2,...,k-1.$$

\noindent
Then $B(2i)=B(2i+1)$, $d_{G'}(u_{2i+1})=2n$ for $2i+1\neq a$ and $d_{G'}(u_{2i+1})=2n-1$ for $2i+1=a$.

\bigskip

$$d_{G'}(u_{2k+1})=m_{G'}(e_{2k})+2\ell_{G'}(u_{2k+1})=$$$$\sum^{k+1}_{j=1} 2\ell_{G'}(u_{2j-1})+A(2k)-B(2k)=2n+A(2k)-B(2k).$$

\noindent
Clearly, $B(2k)=1$, $d_{G'}(u_{2k+1})=2n$ for $a\neq 2k+1$ and $d_{G'}(u_{2k+1})=2n-1$ for $a=2k+1$.

\bigskip

Resulting $G'$ is the connected multiple of $P_{2k+1}$ as required.

\bigskip

\noindent
(III) $a,b$ are odd. From $\pi(H_1)\geq 1$, $\pi(H_{2k+1})\geq 1$ and $\sum^k_{i=0}\pi(H_{2i+1})=n-1$ we have $n>2$.

We define the number of loops at each vertex $u_i$ of multiple $G'$ of $P_{2k+1}$ as in~(I) such that we have $\ell_{G'}(u_1)\geq 1$, $\ell_{G'}(u_{2k+1})\geq 1$, and $\sum^k_{i=0} \ell_{G'}(u_{2i+1})=n-1$.

\bigskip

Now we define the multiplicity of every edge $e_m$ of $P_{2k+1}$, $m=1,2,...,2k$,
$$m_{G'}(e_{2i})=\sum^i_{j=1} 2\ell_{G'}(u_{2j-1})+A(2i)+B(2i) \mbox{~~ and} $$  
$$m_{G'}(e_{2i-1})=2n-\sum^i_{j=1} 2\ell_{G'}(u_{2j-1})-A(2i-1)-B(2i-1), \mbox{ for } i=1,2,...,k.$$

Clearly, by the construction the multiplicity of every edge is at least 2 and the degree of every vertex of $G'$ is the following:
$$d_{G'}(u_{1})=2\ell_{G'}(u_1)+m_{G'}(e_1)=2\ell_{G'}(u_1)+2n-2\ell_{G'}(u_1)-A(1)-B(1)=$$$$2n-A(1)-B(1).$$

\noindent
Then $d_{G'}(u_1)=2n$ for $1\notin\{a,b\}$, $d_{G'}(u_1)=2n-1$, for $1\in\{a,b\}$, $a\neq b$, and $d_{G'}(u_1)=2n-2$ for $1=a=b$.

\bigskip

$$d_{G'}(u_{2i})=m_{G'}(e_{2i-1})+m_{G'}(e_{2i})=$$
$$2n-\sum^i_{j=1} 2\ell_{G'}(u_{2j-1})-A(2i-1)-B(2i-1)+\sum^i_{j=1} 2\ell_{G'}(u_{2j-1})+A(2i)+B(2i)=$$
$$2n-A(2i-1)-B(2i-1)+A(2i)+B(2i)=2n, \mbox{ for } i=1,2,...,k.$$

\noindent
Note that $A(2i-1)=A(2i)$, $B(2i-1)=B(2i)$.

\bigskip

$$d_{G'}(u_{2i+1})=m_{G'}(e_{2i})+m_{G'}(e_{2i+1})+2\ell_{G'}(u_{2i+1})=$$
$$\sum^i_{j=1} 2\ell_{G'}(u_{2j-1})+A(2i)+B(2i)+$$$$2n-\sum^{i+1}_{j=1} 2\ell_{G'}(u_{2j-1})-A(2i+1)-B(2i+1)+2\ell_{G'}(u_{2i+1})=$$
$$2n+A(2i)+B(2i)-A(2i+1)-B(2i+1), \mbox{ for } i=1,2,...,k-1.$$

\noindent
Then $d_{G'}(u_{2i+1})=2n$ for $2i+1\notin\{a,b\}$, $d_{G'}(u_{2i+1})=2n-1$ for $2i+1\in\{a,b\}$, $a\neq b$, and $d_{G'}(u_{2i+1})=2n-2$ for $2i+1=a=b$.

\bigskip

$$d_{G'}(u_{2k+1})=m_{G'}(e_{2k})+2\ell_{G'}(u_{2k+1})=\sum^{k+1}_{j=1} 2\ell_{G'}(u_{2j-1})+A(2k)+B(2k)=$$
$$=2(n-1)+A(2k)+B(2k)=2n-2+A(2k)+B(2k).$$

\noindent
Then $d_{G'}(u_{2k+1})=2n$ for $2k+1\notin\{a,b\}$, $d_{G'}(u_{2k+1})=2n-1$ for $2k+1\in\{a,b\}$, $a\neq b$, and $d_{G'}(u_{2k+1})=2n-2$ for $2k+1=a=b$.

\vskip 2cm

Now let $a=1$ and $b=2k+1$. From $\sum^k_{i=0}\pi(H_{2i+1})=n-1$ we have $n\geq 1$ and again we define the number of loops at each vertex $u_i$ of multiple $G'$ of $P_{2k+1}$ as in (I) such that $\sum^k_{i=0} \ell_{G'}(u_{2i+1})=n-1$. Clearly, $A(j)=1$ for $j=1,2,...,2k+1$, $B(j)=0$ for $j=1,2,...,2k$ and $B(2k+1)=1$.

As in general case, we define specifically the multiplicity of every edge $e_m$ of $P_{2k+1}$, $m=1,2,...,2k$,

$$m_{G'}(e_{2i})=\sum^i_{j=1} 2\ell_{G'}(u_{2j-1})+1 \mbox{ for } i=1,2,...,k,$$  
$$m_{G'}(e_{2i-1})=2n-\sum^i_{j=1} 2\ell_{G'}(u_{2j-1})-1 \mbox{ for } i=1,2,...,k.$$

Clearly, by the construction the multiplicity of every edge is at least 1 and the degree of every vertex of $G'$ is the following:

$$d_{G'}(u_{1})=2\ell_{G'}(u_1)+m_{G'}(e_1)=2\ell_{G'}(u_1)+2n-2\ell_{G'}(u_1)-1=2n-1.$$

\bigskip

$$d_{G'}(u_{2i})=m_{G'}(e_{2i-1})+m_{G'}(e_{2i})=$$$$2n-\sum^i_{j=1} 2\ell_{G'}(u_{2j-1})-1+\sum^i_{j=1} 2\ell_{G'}(u_{2j-1})+1=2n,\mbox{ for } i=1,2,...,k.$$

\vskip 1cm

$$d_{G'}(u_{2i+1})=m_{G'}(e_{2i})+m_{G'}(e_{2i+1})+2\ell_{G'}(u_{2i+1})=$$$$\sum^i_{j=1} 2\ell_{G'}(u_{2j-1})+1+2n-\sum^{i+1}_{j=1} 2\ell_{G'}(u_{2j-1})-1+2\ell_{G'}(u_{2i+1})=2n,$$$$ \mbox{ for } i=1,2,...,k-1.$$

\bigskip

$$d_{G'}(u_{2k+1})=m_{G'}(e_{2k})+2\ell_{G'}(u_{2k+1})=\sum^{k+1}_{j=1} 2\ell_{G'}(u_{2j-1})+1=2(n-1)+1=2n-1.$$

In both cases resulting $G'$ is the connected multiple of $P_{2k+1}$ as required.
\end{proof}

\bigskip

\begin{lemma}
\label{lemmapath2}
Let $P_{2k+1}$ be a path with odd number of vertices, $k\geq 1$, $a,b\in\{1,2,...,2k+1\}$. Let $H_1,H_2,...,H_{2k+1}$ be graphs with $n$ vertices.
Assume that $P_{2k+1}[H_1, H_2,...,H_{2k+1}]$ contains a hamiltonian path $P$ starting in vertex $x$ from $V(P_{2k+1}(\{u_a\})[H_a])$ and ending in a~vertex $y$ from $V(P_{2k+1}(\{u_b\})[H_b])$, where $u_a, u_b\in V(P_{2k+1})$.
\begin{itemize}
\item[(I)] If $a,b$ are even, then $\sum^k_{i=0}\pi(H_{2i+1})\geq n+1$;
\item[(II)] If $a$ is odd and $b$ is even, then $\sum^k_{i=0}\pi(H_{2i+1})\geq n$;
\item[(III)] If $a,b$ are odd, then $\sum^k_{i=0}\pi(H_{2i+1})\geq n-1$.
\end{itemize}
\end{lemma}

\bigskip

\begin{proof}
\noindent (I) $a,b$ are even. Clearly, $A(1)=B(1)=0$ and $A(2k)=B(2k)=1$.

The graph $H_1$ has at least $n-\pi(H_1)$ components. Therefore there are at least $2(n-\pi(H_1))+A(1)+B(1)=2(n-\pi(H_1))$ edges between $H_1$ and $H_2$ in $P$. 

Since $H_2$ has only $n$ vertices, there are at most $2n$ edges in $P$ from $H_2$. Thus, there are at most $2n-2(n-\pi(H_1))-A(2)-B(2)=2\pi(H_1)-A(2)-B(2)$ edges between $H_2$ and $H_3$ in $P$. Note that $A(2)=0$ and $B(2)=0$, if $x,y\notin V(P_{2k+1}(\{u_2\})[H_2])$, respectively.

\bigskip

In general we have at most 
$$\sum^i_{j=1} 2\pi(H_{2j-1})-A(2i)-B(2i), \mbox{ for } i=1,2,...,k,$$
edges between $H_{2i}$ and $H_{2i+1}$ in $P$ and we have at least    
$$2n-\sum^i_{j=1} 2\pi(H_{2j-1})+A(2i-1)+B(2i-1), \mbox{ for } i=1,2,...,k,$$ 
edges between $H_{2i-1}$ and $H_{2i}$ in $P$.

The last subgraph $H_{2k+1}$ has $n-\pi(H_{2k+1})$ components in its path covering. Therefore there are at least $2(n-\pi(H_{2k+1}))$ edges between $H_{2k+1}$ and $H_{2k}$ in $P$. Thus we get: 

$$2(n-\pi(H_{2k+1}))\leq\sum^k_{j=1} 2\pi(H_{2j-1})-A(2k)-B(2k)\leq\sum^k_{j=1} 2\pi(H_{2j-1})-2$$

$$2n+A(2k)+B(2k)=2n+2\leq \sum^k_{j=0} 2\pi(H_{2j+1})$$

$$n+1\leq \sum^k_{j=0} \pi(H_{2j+1}).$$

\vskip 2cm

\noindent (II) $a$ is odd and $b$ is even. Clearly, $B(1)=0$ and $B(2k)=1$.

Similarly as in (I), we have at most 
$$\sum^i_{j=1} 2\pi(H_{2j-1})+A(2i)-B(2i), \mbox{ for } i=1,2,...,k,$$
edges between $H_{2i}$ and $H_{2i+1}$ in $P$ and we have at least    
$$2n-\sum^i_{j=1} 2\pi(H_{2j-1})-A(2i-1)+B(2i-1), \mbox{ for } i=1,2,...,k,$$ 
edges between $H_{2i-1}$ and $H_{2i}$ in $P$.

The last subgraph $H_{2k+1}$ has $n-\pi(H_{2k+1})$ components in its path covering. Therefore there are at least $2(n-\pi(H_{2k+1}))-A(2k+1)+A(2k)$ edges between $H_{2k+1}$ and $H_{2k}$ in~$P$. Clearly, $A(2k+1)=1$. Thus we get: 

$$2(n-\pi(H_{2k+1}))-A(2k+1)+A(2k)\leq\sum^k_{j=1} 2\pi(H_{2j-1})+A(2k)-B(2k)$$

$$2n-A(2k+1)+B(2k)=2n\leq \sum^k_{j=0} 2\pi(H_{2j+1})$$

$$n\leq \sum^k_{j=0} \pi(H_{2j+1}).$$

\bigskip

\noindent (III) $a,b$ are odd. 

Again similarly as in (I), we have at most 
$$\sum^i_{j=1} 2\pi(H_{2j-1})+A(2i)+B(2i), \mbox{ for } i=1,2,...,k,$$
edges between $H_{2i}$ and $H_{2i+1}$ in $P$ and we have at least    
$$2n-\sum^i_{j=1} 2\pi(H_{2j-1})-A(2i-1)-B(2i-1), \mbox{ for } i=1,2,...,k,$$
edges between $H_{2i-1}$ and $H_{2i}$ in $P$.

The last subgraph $H_{2k+1}$ has $n-\pi(H_{2k+1})$ components in its path covering. Therefore there should be at least $2(n-\pi(H_{2k+1}))-A(2k+1)+A(2k)-B(2k+1)+B(2k)$ edges between $H_{2k+1}$ and $H_{2k}$ in $P$. Clearly, $A(2k+1)=B(2k+1)=1$. Thus we get: 

$$2(n-\pi(H_{2k+1}))-A(2k+1)+A(2k)-B(2k+1)+B(2k)\leq~~~~~~~~~~~~~~~~~~~~~~~~~~~~~~~`$$
$$~~~~~~~~~~~~~~~~~~~~~~~~~~~~~~~~~~~~~~~~~~~~~~~~~~~~~~~~~~~~~~~~~~~~~~~~~~~~~~~\sum^k_{j=1} 2\pi(H_{2j-1})+A(2k)+B(2k)$$

$$2n-A(2k+1)-B(2k+1)=2n-2\leq \sum^k_{j=0} 2\pi(H_{2j+1})$$

$$n-1\leq \sum^k_{j=0} \pi(H_{2j+1}).$$
\end{proof}

Now we are ready to prove Theorem \ref{main-odd} statements (ii) and (iii).

\bigskip

\textbf{\emph{(ii)}}
~First suppose that $\sum^k_{i=0}\pi(H_{2i+1})\geq n-1$. Let $x,y$ be vertices in $P_{2k+1}(\{u_1\})[H_1]$, $P_{2k+1}(\{u_{2k+1}\})[H_{2k+1}]$ such that vertices of $H_1,H_{2k+1}$ corresponding to $x,y$ are not vertices of degree 2 in some component (path) of a spanning linear forest of $H_1,H_{2k+1}$ with $\pi(H_1),\pi(H_{2k+1})$ edges, respectively. We show that $P_{2k+1}[H_1, H_2,...,H_{2k+1}]$ contains an $xy$-hamiltonian path.

We set $a=1$ and $b=2k+1$. By Lemma \ref{lemmapath} (III), we find a connected multiple $G'$ of $P_{2k+1}$ such that $d_{G'}(u_l)=2n$ for $l\in\{2,3,...,2k\}$ and $d_{G'}(u_1)=d_{G'}(u_{2k+1})=2n-1$. Note that if $\sum^k_{i=0}\pi(H_{2i+1}) >n-1$ (the multiple $G'$ has more than $n-1$ loops), then we remove arbitrary loops from $G'$ in such a way that $\sum^k_{i=0}\ell_{G'}(u_{2i+1})=n-1$. Clearly, $\ell_{G'}(u_j) \leq \pi(H_j)$ for $j=1,2,...,2k+1$.

As before, we prove that $P_{2k+1}[H_1,H_2,...,H_{2k+1}]$ contains an $xy$-hamiltonian path using exactly $\ell_{G'}(u_i)$ edges of $H_i$ and exactly $m_{G'}(e_j)$ edges between \newline $V(P_{2k+1}(\{u_j\})[H_j])$ and $V(P_{2k+1}(\{u_{j+1}\})[H_{j+1}])$ for $i=1,2,...,2k+1$ and $j=1,2,...,2k$. 

For every vertex $u_i$ and every graph $H_i$, there exists a spanning linear sub\-forest of $P_{2k+1}(\{u_i\})[H_i]\cong H_i$ with components $P_1(u_i),P_2(u_i),...,P_{j_i}(u_i)$ satisfying 

$$j_i=|H_i|-\ell_{G'}(u_i) \hskip 5mm \mbox{ and } \hskip 5mm \sum^{j_i}_{t=1}||P_t(u_i)||=\ell_{G'}(u_i),$$ because $\ell_{G'}(u_i) \leq \pi(H_i)$.

Futhermore after removing all the loops from multiple $G'$, there exists an open eulerian trail $C$ in $G'$ from $u_1$ to $u_{2k+1}$ of $P_{2k+1}$. We obtain the $xy$-hamiltonian path $P$ in $P_{2k+1}[H_1, H_2,...,H_{2k+1}]$ as required by replacing simultaneously the vertices $u_i$ at their $t$-th occurence in $C$ by the component $P_t(u_i)$, for $t=1,2,...,j_i$ and $i=1,2,...,2k+1$, such that $x$ is the first vertex and $y$ is the last vertex of~$P$. Note that $x, y$ are endvertices of different paths $P_t(u_i)$ or isolated vertices.  

\medskip

Now we suppose that $P_{2k+1}[H_1, H_2,...,H_{2k+1}]$ contains some hamiltonian path~$P$. We may assume that the hamiltonian path starts in $H_1$ and ends in $H_{2k+1}$. By Lemma \ref{lemmapath2} (III), we get that $\sum^k_{i=0}\pi(H_{2i+1})\geq n-1.$~~~~~~~~~~~~~~~~~~~~~~~~~~~~~~~~~~~~~~~~~~~~~~~~~~~~~~~~~~~~~~~~~~~~~~$\square$

\bigskip

\textbf{\emph{(iii)}}
~First suppose that $\pi(H_1)\geq 2$, $\pi(H_{2k+1})\geq 2$ and $\sum^k_{i=0}\pi(H_{2i+1})\geq n+1$. Let $x,y$ be vertices in $P_{2k+1}(\{u_a\})[H_a]$, $P_{2k+1}(\{u_b\})[H_b]$, $a,b\in\{1,2,...,2k+1\}$, respectively. We show that $P_{2k+1}[H_1, H_2,...,H_{2k+1}]$ contains an $xy$-hamiltonian path for every $x,y$. 

Suppose that $a,b$ are even. By Lemma \ref{lemmapath} (I), we find a connected multiple $G'$ of $P_{2k+1}$ such that $d_{G'}(u_l)=2n$ for $l\in\{1,2,...,2k+1\}\setminus\{a,b\}$ and either $d_{G'}(u_a)=d_{G'}(u_b)=2n-1$ if $a\neq b$ or $d_{G'}(u_a)=2n-2$ if $a=b$. Note that if $\sum^k_{i=0}\pi(H_{2i+1})>n+1$ (the multiple $G'$ has more than $n+1$ loops), then we remove arbitrary loops from $G'$ in such a way that $\ell_{G'}(u_1)\geq 2$, $\ell_{G'}(u_{2k+1})\geq 2$ and $\sum^k_{i=0} \ell_{G'}(u_{2i+1})=n+1$. Clearly, $\ell_{G'}(u_j) \leq \pi(H_j)$ for $j=1,2,...,2k+1$.

Suppose that $a$ is odd and $b$ is even (the case $a$ is even and $b$ is odd is symmetrical). If $x$ is a vertex of degree 2 in some component (path) $P$ of a spanning linear forest of $H_a$ with $\pi(H_a)$ edges, then we remove one edge of $P$ incident with $x$ from this spanning linear forest. Hence we have $\pi(H_1)\geq 1$, $\pi(H_{2k+1})\geq 1$ and $\sum^k_{i=0}\pi(H_{2i+1})\geq n$. By Lemma \ref{lemmapath} (II), we find a connected multiple $G'$ of $P_{2k+1}$ such that $d_{G'}(u_l)=2n$ for $l\in\{1,2,...,2k+1\}\setminus\{a,b\}$ and $d_{G'}(u_a)=d_{G'}(u_b)=2n-1$. Note that if $\sum^k_{i=0}\pi(H_{2i+1})>n$ (the multiple $G'$ has more than $n$ loops), then we remove arbitrary loops from $G'$ in such a way that $\ell_{G'}(u_1)\geq 1$, $\ell_{G'}(u_{2k+1})\geq 1$ and $\sum^k_{i=0} \ell_{G'}(u_{2i+1})=n$. Clearly, $\ell_{G'}(u_j) \leq \pi(H_j)$ for $j=1,2,...,2k+1$.

Suppose that $a,b$ are odd. We remove at most 2 edges from spanning linear forests of $H_a$ and $H_b$ such that now $x$ and $y$ are not vertices of degree 2 in some component (path) of a spanning linear forest of $H_a$ or $H_b$ and $x,y$ are not in the same component (path) of a spaning linear forest of $H_a=H_b$. Hence we have $\pi(H_1)\geq 1$, $\pi(H_{2k+1})\geq 1$ and $\sum^k_{i=0}\pi(H_{2i+1})\geq n-1$ even if $a,b$ are in the same component. By Lemma \ref{lemmapath} (III), we find a connected multiple $G'$ of $P_{2k+1}$ such that $d_{G'}(u_l)=2n$ for $l\in\{1,2,...,2k+1\}\setminus\{a,b\}$ and either $d_{G'}(u_a)=d_{G'}(u_b)=2n-1$ if $a\neq b$ or $d_{G'}(u_a)=2n-2$ if $a=b$. Note that if $\sum^k_{i=0}\pi(H_{2i+1})>n-1$ (the multiple $G'$ has more than $n-1$ loops), then we remove arbitrary loops from $G'$ in such a way that $\ell_{G'}(u_1)\geq 1$, $\ell_{G'}(u_{2k+1})\geq 1$ and $\sum^k_{i=0} \ell_{G'}(u_{2i+1})=n-1$. Clearly, $\ell_{G'}(u_j) \leq \pi(H_j)$ for $j=1,...,2k+1$.

Similarly as in the previous proof, we prove that $P_{2k+1}[H_1,H_2,...,H_{2k+1}]$ contains an $xy$-hamiltonian path using exactly $\ell_{G'}(u_i)$ edges of $H_i$ and exactly $m_{G'}(e_j)$ edges between $V(P_{2k+1}(\{u_j\})[H_j])$ and $V(P_{2k+1}(\{u_{j+1}\})[H_{j+1}])$ for $i=1,2,...,2k+1$, $j=1,2,...,2k$. 

For every vertex $u_i$ and every graph $H_i$, there exists a spanning linear sub\-forest of $P_{2k+1}(\{u_i\})[H_i]\cong H_i$ with components $P_1(u_i),P_2(u_i),...,P_{j_i}(u_i)$ satisfying 

$$j_i=|H_i|-\ell_{G'}(u_i) \hskip 5mm \mbox{ and } \hskip 5mm \sum^{j_i}_{t=1}||P_t(u_i)||=\ell_{G'}(u_i),$$ because $\ell_{G'}(u_i) \leq \pi(H_i)$.

Futhermore after removing all the loops from multiple $G'$, there exists an open eulerian trail $C$ in $G'$ from $u_a$ to $u_b$ if $a\neq b$ and a closed eulerian trail $C$ in $G'$ if $a=b$. We obtain the $xy$-hamiltonian path in $P_{2k+1}[H_1, H_2,...,H_{2k+1}]$ as required by replacing simultaneously the vertices $u_i$ at their $t$-th occurence in $C$ by the component $P_t(u_i)$ for $t=1,2,...,j_i$, $i=1,2,...,2k+1$ such that $x$ is the first vertex and $y$ is the last vertex of $P$. Note that $x, y$ are endvertices of different paths $P_t(u_i)$ or isolated vertices. 

\medskip

Now suppose that $P_{2k+1}[H_1, H_2,...,H_{2k+1}]$ is hamiltonian connected. Clearly, if $\pi(H_1)\leq 1$, $\pi(H_{2k+1})\leq 1$, then there is no hamiltonian path starting and ending in $H_2$, $H_{2k}$, respectively. Thus $\pi(H_1)\geq 2$, $\pi(H_{2k+1}) \geq 2$ and $n>2$ ($P_{2k+1}[H_1, H_2,...,H_{2k+1}]$ has to be 3-connected). Since this graph has a hamilonian path between two arbitrary vertices, by Lemma \ref{lemmapath2}, we immediately get that $\sum^k_{i=0}\pi(H_{2i+1})\geq n+1.$~~~~~~~~~~~~~~~~~~~~~~~~~~~~~~~~~~~~~~~~~~~~~~~~~~~~~~~~~$\square$

\bigskip

\noindent
\textbf{\emph{Proof of Theorem \ref{main-even}}}

\textbf{\emph{(i)}} First assume that the graph $P_{2k}[H_1, H_2,...,H_{2k}]$ has hamiltonian cycle~$C$ and $\pi(H_1) =0$. Since every vertex in $C$ has degree 2 and there is no edge in $E(H_1)$, there are exactly $2n$ edges between $H_1$ and $H_2$. Thus there is no edge of~$C$ between $H_2$ and $H_3$ and we get that $k=1$. Similarly for $\pi(H_{2k}) =0$.

Now assume that $k=1$ or $\pi(H_1)\geq 1$ and  $\pi(H_{2k})\geq 1$. Then the hamiltoni\-city of $P_{2k}[H_1, H_2,...,H_{2k}]$ follows immediately from the proof of Theorem 6 from~\cite{Kri}. The author proved in \cite{Kri} the hamiltonicity of lexicographic product of $G[H]$ where $G$ is traceable by finding a $2n$-regular multiple of $P_{2k}$ which uses only one loop at the first and last vertex of $P_{2k}$. 

\textbf{\emph{(ii)}} Again, the proof is an easy consequence of the proof of Theorem 6 from~\cite{Kri}.

\textbf{\emph{(iii)}} Let $k=1$. If $\pi(H_1)=0$ or $\pi(H_{2})=0$, then clearly, there is no hamiltonian path between some vertices of $H_2$ or $H_{1}$, respectively. If $\pi(H_1)\geq 1$ and $\pi(H_{2k}) \geq 1$, then clearly there exists a hamiltonian path between every two vertices of $P_{2}[H_1, H_2]$.   
 
Let $k>1$. 
First suppose that $P_{2k}[H_1, H_2,...,H_{2k}]$ is hamiltonian connected. If $\pi(H_1) \leq 1$ or $\pi(H_{2k}) \leq 1$, then clearly, there is no hamiltonian path between some two vertices of $H_2$ or $H_{2k-1}$, respectively.

Now assume that $\pi(H_1)\geq 2$ and  $\pi(H_{2k})\geq 2$. Let $x$ be a vertex of \newline $V(P_{2k}(\{u_a\})[H_a])$ and $y$ be a vertex of $V(P_{2k}(\{u_b\})[H_b])$. We may assume that $a \leq b$. Let $F$ be 1-factor of $P_{2k}$, $S_1$ a multigraph with only one vertex $u_1$ and one loop, $S_2$ a multigraph with only vertex $u_{2k}$ and one loop, and let $P' \subseteq P_{2k}$ be a path from $u_a$ to $u_b$ in $P_{2k}$ and $P'' \subseteq P_{2k}$ a path from $u_b$ to $u_{2k}$ in $P_{2k}$ if any. 

\begin{itemize}
\item[(I)] Let $b-a$ be odd. Then $P'$ is an even path and let $F_1$ be 1-factor of $P'$ and set $F_2=P' -E(F_1)$. Then we define multiple $G'$:
$$G'=2P_{2k}+(2n-4)F+S_1+S_2-F_1+F_2.$$
Clearly, the degree of every vertex of $G'$ except from $u_a$ and $u_b$ is $2n$ and $d_{G'}(u_a)=d_{G'}(u_b)=2n-1$ and the multiplicity of every edge of $G'$ is at least 1. 

\item[(II)] Let $b-a$ be even. We may assume that $a$ is odd (otherwise we relabel all the vertices: $u_1\rightarrow u_{2k}, u_2\rightarrow u_{2k-1},...,u_{2k}\rightarrow u_1$). Let $F_1$ be 1-factor of $P'-u_b$ and $F_2$ be 1-factor of $P'-u_a$. Since $a$ is odd, $b$ is odd as well. Thus $P''$ is an even path. Let $F_3$ be 1-factor of $P''$ and set $F_4=P'' -E(F_3)$. Then we define multiple $G'$:
$$G'=2P_{2k}+(2n-4)F+S_1+2S_2-F_1+F_2-2F_3+2F_4.$$
Clearly, the degree of every vertex of $G'$ except from $u_a$ and $u_b$ is $2n$, $d_{G'}(u_a)=d_{G'}(u_b)=2n-1$ if $a\neq b$, $d_{G'}(u_a)=d_{G'}(u_b)=2n-2$ if $a=b$, and the multiplicity of every edge of $G'$ is at least 1. Note that $n\geq 3$ because of $\pi(H_1)\geq 2$.

\end{itemize}

From such a multiple $G'$ in both cases we get a hamiltonian path between two arbitrary vertices of the graph $P_{2k}[H_1, H_2,...,H_{2k}]$ similarly as in the proof of Theorem~\ref{main-odd}.~~~~~~~~~~~~~~~ $\square$

\bigskip

\section{Next results}
Now we easily get results concerning the lexicographic product of $P_{2k+1}$ and given graph~$H$.  

\begin{theorem}
\label{hamcycle}
Let $P_{2k+1}$ be a path with $2k+1$ vertices, $k\geq 1$, and $H$ be a graph. Then the lexicographic product $P_{2k+1}[H]$ is pancyclic if and only if $\pi(H)\geq 1$ and $\lceil\frac{|H|}{\pi(H)}\rceil\leq k+1$.
\end{theorem}

\begin{proof}
Set $H_1=\cdots=H_{2k+1}=H$. By Theorem \ref{main-odd} (i), $P_{2k+1}[H]$ is hamiltonian \hspace{2mm} $\Leftrightarrow$ \hspace{2mm}                                $\sum^k_{i=0}\pi(H_{2i+1})\geq |H|$\hspace{2mm} $\Leftrightarrow$ \hspace{2mm} $(k+1)\pi(H)\geq |H|$\hspace{2mm} $\Leftrightarrow$ \hspace{2mm}$\lceil\frac{|H|}{\pi(H)}\rceil\leq k+1$ because $k+1$ is an integer. The pancyclicity follows from Theorem \ref{Pancyclicity}.   
\end{proof}

\medskip

\begin{theorem} 
\label{hampath}
Let $P_{2k+1}$ be a path with $2k+1$ vertices, $k\geq 1$, and $H$ be a graph. Then the lexicographic product $P_{2k+1}[H]$ is traceable if and only if $\lceil\frac{|H|-1}{\pi(H)}\rceil\leq k+1$.
\end{theorem}

\begin{proof}
Set $H_1=\cdots=H_{2k+1}=H$. By Theorem \ref{main-odd} (ii), $P_{2k+1}[H]$ is traceable \hspace{2mm} $\Leftrightarrow$ \hspace{2mm}$\sum^k_{i=0}\pi(H_{2i+1})\geq |H|-1$ \hspace{1mm} $\Leftrightarrow$\hspace{1mm} $(k+1)\pi(H)\geq |H|-1$ \hspace{1mm} $\Leftrightarrow$ \hspace{1mm}$\lceil\frac{|H|-1}{\pi(H)}\rceil\leq k+1$ because $k+1$ is an integer.   
\end{proof}

\medskip

\begin{corollary}
Let $G$ and $H$ be graphs, $|G|\geq 2$. Suppose that $G$ contains a hamiltonian path.

\begin{itemize}

\item If $|G|=2k+1$ and $\frac{|H|}{k+1}\leq\pi(H)$, then $G[H]$ is hamiltonian.

\item If $|G|=2k+1$ and $\frac{|H|-1}{k+1}\leq\pi(H)$, then $G[H]$ is traceable.

\end{itemize}
\end{corollary}

Thus we improved Theorem \ref{KriTei} and the bounds are the best possible. Moreover we get a similar result also for hamiltonian connectivity.

\bigskip

\begin{theorem}
\label{hamcon-odd}
Let $P_{2k+1}$ be a path with $2k+1$ vertices, $k\geq 1$, and $H$ be a graph. Then the lexicographic product $P_{2k+1}[H]$ is hamiltonian connected if and only if $\pi(H)\geq 2$ and $\lceil\frac{|H|+1}{\pi(H)}\rceil\leq k+1$.
\end{theorem}

\begin{proof}
Set $H_1=\cdots=H_{2k+1} = H$. By Theorem \ref{main-odd} (iii), $P_{2k+1}[H]$ is hamiltonian connected \hspace{2mm} $\Leftrightarrow$ \hspace{2mm}$\sum^k_{i=0}\pi(H_{2i+1})\geq |H|+1$ \hspace{1mm}                                     
$\Leftrightarrow$\hspace{1mm} $(k+1)\pi(H)\geq |H|+1$ \hspace{1mm} $\Leftrightarrow$ \hspace{1mm}$\lceil\frac{|H|+1}{\pi(H)}\rceil\leq k+1$ because $k+1$ is an integer.   
\end{proof}

\begin{corollary}
Let $G$ and $H$ be graphs, $|G|\geq 2$. Suppose that $G$ contains a hamiltonian path.

If $|G|=2k+1$ and $\frac{|H|+1}{k+1}\leq\pi(H)$, then $G[H]$ is hamiltonian connected.
\end{corollary}

\begin{theorem}
\label{hamcon-even}
Let $P_{2k}$ be a path with $2k$ vertices, $k\geq 1$, and $H$ be a graph. Then the lexicographic product $P_{2k}[H]$ is hamiltonian connected if and only 
      \begin{itemize}
         \item[$\bullet$] $\pi(H)\geq 1$ for $k=1$.
         \item[$\bullet$] $\pi(H)\geq 2$ for $k>1$.
      \end{itemize}
\end{theorem}

\begin{proof}
Easy corollary of Theorem \ref{main-even} (iii).
\end{proof}

\begin{corollary}
Let $G$ and $H$ be graphs, $|G|\geq 2$. Suppose that $G$ contains a hamiltonian path.

\begin{itemize}
\item[] If $|G|=2$ and $||H||\geq 1$, then $G[H]$ is hamiltonian connected.
\item[] If $|G|=2k$, for $k>1$, and $||H||\geq 2$, then $G[H]$ is hamiltonian connected.
\end{itemize}

\end{corollary}

\section{Conclusion}
In this paper we finished a complete characterization of hamiltonicity (Theorem~\ref{KriTei} and Theorem \ref{hamcycle}), traceability (Theorem \ref{KriTei} and Theorem \ref{hampath}) and hamiltonian connectedness (Theorem \ref{hamcon-odd} and Theorem \ref{hamcon-even}) of $G[H]$, where $G$ is a path. Hence we improved and extended results in \cite{Kri}. Moreover we proved these results also for lex-regular generalized lexicographic products. If $G$ has no hamiltonian path, then for general graphs $G$ it seems to be complicated to characterize when $G[H]$ is traceable, hamiltonian or hamiltonian connected. Let us mention that Kaiser and Kriesell proved in \cite{KaiKri} that if $G$ is 4-tough and $||H||\geq 1$, then $G[H]$ is hamiltonian. Since $G$ is 4-tough implies that $G$ has a 2-walk, $G$ is not so far from being hamiltonian. Clearly,
 $$\mbox{hamiltonicity} \Longrightarrow \mbox{traceability} \Longrightarrow \mbox{2-walk}  
 \Longrightarrow \mbox{3-tree} \Longrightarrow \mbox{3-walk} \Longrightarrow ....$$ 
Hence it could be interesting to study hamiltonian paths and cycles in $G[H]$ if G is a 3-tree. 
 
 \bigskip
 
 This work was partly supported by the European Regional Development Fund (ERDF), project NTIS - New Technologies for Information Society, European Centre of Excellence, CZ.1.05/1.1.00/02.0090, and by project GA20-09525S of the Czech Science Foundation.


\bigskip 
 
\begin{thebibliography}{99} 
\bibitem{BarSz} Z. Baranyai, G. R. Sz\'asz; Hamiltonian Decomposition of Lexicographic Product; Journal of Combinatorial Theory, Series B 31 (1981), 253-261.

\bibitem{Bon} J.A. Bondy, U.S.R. Murty; Graph Theory; Graduate Texts in Mathematics 244; Springer, New York 2008.    

\bibitem{GuHou} R. Gu, H. Hou; End-regular and End-orthodox generalized lexicographic products of bipartite graphs; Open Mathematics 14 (2016), 229-236.    

\bibitem{ChoKar} S. A. Choudum, T. Karthick; Maximal cliques in $\{P_2\cup P_3,C_4\}$-free graphs; Discrete Mathematics 310 (2010), 3398-3403.

\bibitem{KaiKri} T. Kaiser, M. Kriesell; On the Pancyclicity of Lexicographic Products; Graphs and Combinatorics 22 (1) (2006), 51-58.

\bibitem{Kri} M. Kriesell; A Note on Hamiltonian Cycles in Lexicographical Products; Journal of Automata, Languages and Combinatorics 2 (2) (1997), 135-138.

\bibitem{Ng} L. L. Ng; Hamiltonian Decomposition of Lexicographic Products of Digraphs; Journal of Combinatorial Theory, Series B 73 (1998), 119-129.

\bibitem{Sam} V. Samodivkin; Domination related parameters in the generalized lexicographic product of graphs; Discrete Applied Mathematics 300 (2021), 77-84.

\bibitem{Tei} H.-M. Teichert; Hamiltonian properties of the lexicographic product of undirected graphs; Elektronische Informationsverarbeitung Kybernetik 19 (1983) 1/2, 67-77.

\end{thebibliography}
\end{document}